\documentclass[11pt]{article}

\usepackage{geometry}
\usepackage{amsmath,amssymb,amsthm}

\def\er{\mathbb{R}}
\def\Ex{\mathbb{E}}
\def\Pr{\mathbb{P}}
\def\ve{\varepsilon}
\def\calm{\mathcal{M}}
\def\calz{\mathcal{Z}}
\def\cale{\mathcal{E}}

\newtheorem{thm}{Theorem}
\newtheorem{lem}[thm]{Lemma}
\newtheorem{prop}[thm]{Proposition}
\newtheorem{cor}[thm]{Corollary}

\theoremstyle{remark}
\newtheorem{rem}[thm]{Remark}

\title{On $\calz_p$-norms of random vectors}
\author{Rafa{\l} Lata{\l}a
\thanks{Supported by the National Science Centre, Poland grant 2015/18/A/ST1/00553}}
\date{}

\begin{document}

\maketitle

\begin{abstract}
To any $n$-dimensional random vector $X$ we may associate its $L_p$-centroid body $\calz_p(X)$ and the 
corresponding norm. We formulate a conjecture concerning the bound on the $\calz_p(X)$-norm of $X$ and
show that it holds under some additional symmetry assumptions. We also relate our conjecture with estimates
of covering numbers and Sudakov-type minorization bounds.
\end{abstract}

\section{Introduction. Formulation of the Problem.}

Let $p\geq 2$ and $X=(X_1,\ldots,X_n)$ be a random vector in $\er^n$ such that $\Ex|X|^p<\infty$.
We define the following two norms on $\er^n$:
\[
\|t\|_{\calm_p(X)}:=(\Ex|\langle t,X\rangle|^p)^{1/p}
\quad\mbox{and}\quad
\|t\|_{\calz_p(X)}:=\sup\{|\langle t,s\rangle|\colon\ \|s\|_{\calm_p(X)}\leq 1\}.
\]
By $\calm_p(X)$ and $\calz_p(X)$ we will also denote unit balls in these norms, i.e.
\[
\calm_p(X):=\{t\in \er^n\colon\ \|t\|_{\calm_p(X)}\leq 1\}
\mbox{ and }
\calz_p(X):=\{t\in \er^n\colon\ \|t\|_{\calz_p(X)}\leq 1\}.
\]

The set $\calz_p(X)$ is called the $L_p$-centroid body of $X$ (or rather of the distribution of $X$).
It was introduced (under a different normalization) for uniform distributions on convex bodies in \cite{LZ}. 
Investigation of $L_p$-centroid bodies played a crucial role in the Paouris proof of large deviations bounds for 
Euclidean norms of log-concave vectors \cite{Pa}. Such bodies also appears in questions related to the optimal concentration of log-concave vectors \cite{La2}.

Let us introduce a bit of useful notation. 
We set $|t|:=\|t\|_2=\sqrt{\langle t,t\rangle}$ and $B_2^n=\{t\in \er^n\colon |t|\leq 1\}$. 
By  $\|Y\|_p=(\Ex|Y|^p)^{1/p}$ we denote the $L_p$-norm of a random variable $Y$.
Letter $C$ denotes universal constants (that may differ at each occurence), we write $f\sim g$ if
$\frac{1}{C}f\leq g\leq Cf$.

Let us begin with a simple case, when a random vector $X$ is rotationally invariant. Then $X=RU$, 
where $U$ has a uniform distribution on $S^{n-1}$  and $R=|X|$ is a nonnegative random variable, independent of $U$.
We have for any vector $t\in \er^n$ and $p\geq 2$,
\[
\|\langle t,U\rangle\|_p=|t|\|U_1\|_p\sim \sqrt{\frac{p}{n+p}}|t|,
\]
where $U_1$ is the first coordinate of $U$.
Therefore
\[
\|t\|_{\calm_p(X)}=\|U_1\|_p\|R\|_p|t|%%\sim  \sqrt{\frac{p}{n+p}}\|R\|_p|t|
\quad \mbox{and} \quad
\|t\|_{\calz_p(X)}=\|U_1\|_p^{-1}\|R\|_p^{-1}|t|.
\]
So
\begin{equation}
\label{eq:rotinv}
\left(\Ex\|X\|_{\calz_p(X)}^p\right)^{1/p}=\|U_1\|_p^{-1}\|R\|_p^{-1}(\Ex|X|^p)^{1/p}
=\|U_1\|_p^{-1}\sim \sqrt{\frac{n+p}{p}}.
\end{equation}

This motivates the following problem.
\medskip

\noindent
{\bf Problem 1.} Is it true that for (at least a large class of) centered $n$-dimensional random vectors $X$, 
\[
\left(\Ex \|X\|_{\calz_p(X)}^2\right)^{1/2}\leq C\sqrt{\frac{n+p}{p}} \quad \mbox{ for }p\geq 2,
\]
or maybe even
\[
\left(\Ex \|X\|_{\calz_p(X)}^p\right)^{1/p}\leq C\sqrt{\frac{n+p}{p}} \quad \mbox{ for }p\geq 2?
\]

\medskip

Notice that the problem is linearly-invariant, since 
\begin{equation}
\label{eq:inv}
\|AX\|_{{\cal Z}_p(AX)}=\|X\|_{{\cal Z}_p(X)}
\quad \mbox{ for any }A\in \mathrm{GL}(n).
\end{equation}

For any centered random vector $X$ with nondegenerate covariance matrix, random vector $Y=\mathrm{Cov}(X)^{-1/2}X$ is
isotropic (i.e. centered with identity covariance matrix). We have $\calm_2(Y)=\calz_2(Y)=B_2^n$, hence
\[
\Ex\|X\|_{\calz_2(X)}^2=\Ex\|Y\|_{\calz_2(Y)}^2=\Ex|Y|^2=n.
\] 
Next remark shows that the answer to our problem is positive in the case $p\geq n$.

\begin{rem}
\label{rem:largen}
For $p\geq n$ and any $n$-dimensional random vector $X$ we have $(\Ex\|X\|_{\calz_p(X)}^p)^{1/p}\leq 10$.
\end{rem}

\begin{proof}
Let $S$ be a $1/2$-net in the unit ball of $\calm_p(X)$ such that $|S|\leq 5^n$ (such net exists by the 
volume-based argument, cf. \cite[Corollary 4.1.15]{AGM}). Then
\begin{align*}
(\Ex\|X\|_{\calz_p(X)}^p)^{1/p}
&\leq 2\left(\Ex\sup_{t\in S}|\langle t,X\rangle|^p\right)^{1/p}
\leq 2\left(\Ex\sum_{t\in S}|\langle t,X\rangle|^p\right)^{1/p}
\\
&\leq 2|S|^{1/p}\sup_{t\in S}(\Ex\langle t,X\rangle|^p)^{1/p}
\leq 2\cdot 5^{n/p}.
\end{align*}
\end{proof}

$L_p$-centroid bodies play an important role in the study of vectors uniformly distributed on convex bodies and a more general class of log-concave vectors. A random vector with a nondenerate covariance matrix is called log-concave if its density has the form $e^{-h}$, where $h\colon\ \er^n\to (-\infty,\infty]$ is convex. If $X$ is 
centered and log-concave then
\begin{equation}
\label{eq:momlogc}
\|\langle t,X\rangle\|_p\leq \lambda \frac{p}{q}\|\langle t,X\rangle\|_q\quad \mbox{ for }p\geq q\geq 2,
\end{equation}
where  $\lambda=2$ ($\lambda=1$ if $X$ is symmetric and log-concave and $\lambda=3$ for arbitrary log-concave vectors). One of open problems for log-concave vectors \cite{La2} states that for such vectors, arbitrary norm
$\|\ \|$ and $q\geq 1$,
\[
(\Ex \|X\|^q)^{1/q}\leq C\left(\Ex\|X\|+\sup_{\|t\|_*\leq 1}\|\langle t,X\rangle\|_q\right).
\]
In particular one may expect that for log-concave vectors
\[
(\Ex \|X\|_{\calz_p(X)}^q)^{1/q}\leq C\left(\Ex\|X\|_{\calz_p(X)}+\sup_{t\in \calm_p(X)}\|\langle t,X\rangle\|_q\right)
\leq C\left(\Ex\|X\|_{\calz_p(X)}+\frac{\max\{p,q\}}{p}\right).
\]

As a result it is natural to state the following variant of Problem 1.
\medskip

\noindent
{\bf Problem 2.} Let $X$ be a centered log-concave $n$-dimensional random vector. Is it true that
\[
(\Ex\|X\|_{\calz_p(X)}^q)^{1/q}\leq C\sqrt{\frac{n}{p}}\quad \mbox{ for }2\leq p\leq n,\ 1\leq q\leq \sqrt{pn}.
\]

In Section 2 we show that Problems 1 and 2 have affirmative solutions in the class of unconditional vectors.
In Section 3 we relate our problems to estimates of covering numbers. We also show that the first estimate in Problem 1 holds if the random vector $X$ satisfies the Sudakov-type minorization bound.

\section{Bounds for unconditional random vectors}

In this section we consider the class of \emph{unconditional} random vectors in $\er^n$, 
i.e. vectors $X$ having the same distribution as $(\ve_1 |X_1|,\ve_2 |X_2|,\ldots,\ve_n |X_n|)$, where
$(\ve_i)$ is a sequence of independent symmetric $\pm 1$ random variables (Rademacher sequence), independent of $X$.

Our first result shows that formula \eqref{eq:rotinv}  may be extended  to the unconditional case for $p$ even.
We use the standard notation -- for a multiindex $\alpha=(\alpha_1,\ldots,\alpha_n)$, $x\in \er^n$ and
$m=\sum \alpha_i$,
$x^{\alpha}:=\prod_i x_i^{\alpha_i}$ and $\binom{m}{\alpha}:=m!/(\prod_i \alpha_i!)$.

\begin{prop}
\label{prop:unceven}
We have for any $k=1,2,\ldots$ and any $n$-dimensional unconditional random vector $X$ such that 
$\Ex|X|^{2k}<\infty$,
\[
\left(\Ex\|X\|_{\calz_{2k}(X)}^{2k}\right)^{1/(2k)}\leq c_{2k}
:=\left(\sum_{\|\alpha\|_1=k}\frac{\binom{k}{\alpha}^2}{\binom{2k}{2\alpha}}\right)^{1/(2k)}
\sim \sqrt{\frac{n+k}{k}},
\]
where the summation runs over all multiindices $\alpha=(\alpha_1,\ldots,\alpha_n)$ with nonnegative
integer coefficients such that $\|\alpha\|_1=\sum_{i=1}^n\alpha_i=k$.
\end{prop}

\begin{proof}
Observe first that
\[
\Ex|\langle t,X\rangle|^{2k}=\Ex\left|\sum_{i=1}^n t_i\ve_iX_i\right|^{2k}
=\sum_{\|\alpha\|_1=k}\binom{2k}{2\alpha}t^{2\alpha}\Ex X^{2\alpha}.
\]
For any $t,s\in \er^n$ we have
\[
|\langle t,s\rangle|^{k}
=\sum_{\|\alpha\|_1=k}\binom{k}{\alpha}t^{\alpha}s^{\alpha}.
\]
So by the Cauchy-Schwarz inequality,
\[
\|s\|_{\calz_{2k}(X)}^{k}=\sup\{|\langle t,s\rangle|^{k}\colon \Ex|\langle t,X\rangle|^{2k}\leq 1\}
\leq \left(\sum_{\|\alpha\|_1=k}\frac{\binom{k}{\alpha}^2}{\binom{2k}{2\alpha}}
\frac{s^{2\alpha}}{\Ex X^{2\alpha}}\right)^{1/2}.
\]

To see that $c_{2k}\sim \sqrt{(n+k)/k}$ observe that
\[
\frac{\binom{k}{\alpha}^2}{\binom{2k}{2\alpha}}=
\binom{2k}{k}^{-1}\prod_{i=1}^{n}\binom{2\alpha_i}{\alpha_i}.
\]
Therefore, since $1\leq \binom {2l}{l}\leq 2^{2l}$, we get
\[
4^{-k}\binom{n+k-1}{k}\leq c_{2k}^{2k}\leq 4^k \binom{n+k-1}{k}.
\]
\end{proof}

\begin{cor}
Let $X$ be an unconditional $n$-dimensional random vector. Then 
\[
\left(\Ex\|X\|_{\calz_p(X)}^{2k}\right)^{1/2k}\leq C\sqrt{\frac{n+p}{p}}\quad 
\mbox{ for any positive integer }k\leq \frac{p}{2}.
\]
\end{cor}

\begin{proof}
By the monotonicity of $L_{2k}$-norms we may and will assume that $k=\lfloor p/2\rfloor$.
Then by Proposition \ref{prop:unceven},
\[
\left(\Ex\|X\|_{\calz_p(X)}^{2k}\right)^{1/2k}\leq
\left(\Ex\|X\|_{\calz_{2k}(X)}^{2k}\right)^{1/2k}\leq C\sqrt{\frac{n+k}{k}}\leq C\sqrt{\frac{n+p}{p}}.
\]
\end{proof}

In the unconditional log-concave case we may bound higher moments of $\|X\|_{\calz_p(X)}$. 

\begin{thm}
\label{thm:unclogcon}
Let $X$ be an unconditional log-concave $n$-dimensional random vector. Then for $p,q\geq 2$,
\[
(\Ex\|X\|_{\calz_p(X)}^q)^{1/q}\leq 
C\left(\sqrt{\frac{n+p}{p}}+\sup_{t\in \calm_p(X)}\|\langle t,X\rangle\|_q\right)
\leq C\left(\sqrt{\frac{n+p}{p}}+\frac{q}{p}\right).
\]
\end{thm}

In order to show this result we will need the following lemma.

\begin{lem}
\label{lem:estzpunc}
Let $2\leq p\leq n$, $X$ be an unconditional random vector in $\er^n$ such that $\Ex|X|^p<\infty$ and $\Ex |X_i|=1$. 
Then
\begin{equation}
\label{eq:estzpunc}
\|s\|_{\calz_p(X)}
\leq \sup_{I\subset [n],|I|\leq p}\sup_{\|t\|_{\calm_p(X)}\leq 1}\left|\sum_{i\in I}t_is_i\right|+
C_1\sup_{\|t\|_{{\cal M}_p(X)}\leq 1,\|t\|_2\leq p^{-1/2}}\left|\sum_{i=1}^n t_is_i\right|.
\end{equation}
\end{lem}

\begin{proof}
We have by the unconditionality of $X$ and Jensen's inequality,
\[
\|t\|_{{\cal M}_p(X)}=\left\|\sum_{i=1}^n t_i \ve_i |X_i|\right\|_p\geq 
\left\|\sum_{i=1}^n t_i \ve_i \Ex |X_i|\right\|_p.
\]

By the result of Hitczenko \cite{H}, for numbers $a_1,\ldots,a_n$,
\begin{equation}
\label{eq:H}
\left\|\sum_{i=1}^n a_i \ve_i \right\|_p\sim \sum_{i\leq p}a_i^*+\sqrt{p}\left(\sum_{i>p}|a_i^*|^2\right)^{1/2},
\end{equation}
where $(a_i^*)_{i\leq n}$ denotes the nonincreasing rearrangement of $(|a_i|)_{i\leq n}$.
Thus
\[
\sqrt{p}\left(\sum_{i>p}|t_i^*|^2\right)^{1/2}\leq C_1\|t\|_{\calm_p(X)}
\]
and \eqref{eq:estzpunc} easily follows.
\end{proof}

\begin{proof}[Proof of Theorem \ref{thm:unclogcon}]

The last bound in the assertion follows by \eqref{eq:momlogc}.
It is easy to see that (increasing $q$ if necessary) it is enough to consider the case $q\geq \sqrt{np}$. 

If $q\geq n$ then the similar argument as in the proof of Remark \ref{rem:largen} shows that 
\[
\left(\Ex\|X\|_{\calz_p(X)}^q\right)^{1/q}\leq 2\cdot 5^{n/q}\sup_{t\in \calm_p(X)}\|\langle t,X\rangle\|_q
\leq 10 \sup_{t\in \calm_p(X)}\|\langle t,X\rangle\|_q.
\]

Finally, consider the remaining case $\sqrt{pn}\leq q\leq n$.
By \eqref{eq:inv} we may assume that $\Ex |X_i|=1$ for all $i$. 
By the log-concavity $\|\langle t,X\rangle \|_{q_1}\leq C\frac{q_1}{q_2}\|\langle t,X\rangle\|_{q_2}$ for $q_1\geq q_2\geq 1$, 
in particular $\sigma_i:=\|X_i\|_2\leq C$.

Let $\cale_1,\ldots,\cale_n$ be i.i.d. symmetric exponential random variables with variance $1$. By \cite[Theorem 3.1]{La} we have
\begin{align*}
&\left\|\sup_{\|t\|_{\calm_p(X)}\leq 1,\|t\|_2\leq p^{-1/2}}\left|\sum_{i=1}^n t_iX_i\right|\right\|_q
\\
&\phantom{aaaaaa}
\leq 
C\left(\left\|\sup_{\|t\|_{\calm_p(X)}\leq 1,\|t\|_2\leq p^{-1/2}}\left|\sum_{i=1}^nt_i\sigma_i\cale_i\right|\right\|_1
+\sup_{\|t\|_{\calm_p(X)}\leq 1,\|t\|_2\leq p^{-1/2}}\|\langle t,X\rangle\|_q
\right).
\end{align*}
We have
\[
\sup_{\|t\|_{\calm_p(X)}\leq 1,\|t\|_2\leq p^{-1/2}}\|\langle t,X\rangle\|_q\leq 
\sup_{\|t\|_{\calm_p(X)}\leq 1}\|\langle t,X\rangle\|_q
\]
and
\[
\left\|\sup_{\|t\|_{\calm_p(X)}\leq 1,\|t\|_2\leq p^{-1/2}}\left|\sum_{i=1}^nt_i\sigma_i\cale_i\right|\right\|_1
\leq \frac{1}{\sqrt{p}}\left\|\sqrt{\sum_{i=1}^n\sigma_i^2\cale_i^2}\right\|_1
\leq \frac{1}{\sqrt{p}}\sqrt{\sum_{i=1}^n\sigma_i^2}\leq C\sqrt{\frac{n}{p}}.
\]
Thus 
\[
\left\|\sup_{\|t\|_{\calm_p(X)}\leq 1,\|t\|_2\leq p^{-1/2}}\left|\sum_{i=1}^n t_iX_i\right|\right\|_q\leq 
C\left(\sqrt{\frac{n}{p}}+\sup_{\|t\|_{\calm_p(X)}\leq 1}\|\langle t,X\rangle\|_q\right).
\]

Let for each $I\subset[n]$, $P_IX=(X_i)_{i\in I}$ and $S_I$ be a $1/2$-net in $\calm_p(P_IX)$ of
cardinality at most $5^{|I|}$. We have
\begin{align*}
\left\|\sup_{I\subset [n],|I|\leq p}\sup_{\|t\|_{\calm_p(X)}\leq 1}\left|\sum_{i\in I}t_iX_i\right|\right\|_q
&\leq 2\left\|\sup_{I\subset [n],|I|\leq p}\sup_{t\in S_I}\left|\sum_{i\in I}t_iX_i\right|\right\|_{q}
\\
&\leq 2\left(\sum_{I\subset [n],|I|\leq p}\sum_{t\in S_I}\Ex\left|\sum_{i\in I}t_iX_i\right|^{q}\right)^{1/q}
\\
&\leq 2\cdot 5^{p/q}|\{I\subset [n], |I|\leq p\}|^{1/q}
\sup_{I}\sup_{t\in S_I}\left\|\sum_{i\in I}t_iX_i\right\|_{q}
\\
&\leq 10 \left(\frac{en}{p}\right)^{p/q}\sup_{t\in \calm_p(X)}\left\|\sum_{i\in I}t_iX_i\right\|_{q}
\\
&\leq C\sup_{t\in \calm_p(X)}\left\|\sum_{i\in I}t_iX_i\right\|_{q},
\end{align*}
where the last estimate follows from $q\geq \sqrt{np}$.

Hence the assertion follows by Lemma \ref{lem:estzpunc}.
\end{proof}

\begin{cor}
\label{cor:unclogcon}
Let $X$ be an unconditional log-concave $n$-dimensional random vector and $2\leq p\leq n$. Then 
\begin{equation}
\label{eq:momunclogconc}
\frac{1}{C}\sqrt{\frac{n}{p}}\leq \Ex \|X\|_{\calz_p(X)}
\leq \left(\Ex  \|X\|_{\calz_p(X)}^{\sqrt{np}}\right)^{1/\sqrt{np}}\leq C\sqrt{\frac{n}{p}}
\end{equation}
and
\[
\Pr\left(\|X\|_{\calz_p(X)}\geq \frac{1}{C}\sqrt{\frac{n}{p}}\right)\geq \frac{1}{C},\quad
\Pr\left(\|X\|_{\calz_p(X)}\geq Ct\sqrt{\frac{n}{p}}\right)\leq e^{-t\sqrt{np}} \mbox{ for }t\geq 1.
\]
\end{cor}

\begin{proof}
The upper bound in \eqref{eq:momunclogconc} easily follows by
Theorem \ref{thm:unclogcon}. In fact we have for $t\geq 1$,
\[
\left(\Ex  \|X\|_{\calz_p(X)}^{t\sqrt{np}}\right)^{1/(t\sqrt{np})}\leq Ct\sqrt{\frac{n}{p}},
\]
hence the Chebyshev inequality yields the upper tail bound for $\|X\|_{\calz_p(X)}$.

To establish lower bounds we may assume that $X$ is additionally isotropic. Then by the result of 
Bobkov and Nazarov \cite{BN} we have 
$\|\langle t,X\rangle \|_p\leq C(\sqrt{p}\|t\|_2+p\|t\|_{\infty})$. This easily gives
\[
\Ex \|X\|_{\calz_p(X)}\geq \frac{1}{C}\sqrt{\frac{n}{p}}\Ex X_{\lceil n/2\rceil }^*
\geq \frac{1}{C}\sqrt{\frac{n}{p}},
\]
where the last inequality follows by Lemma \ref{lem:esthalf} below.

By the Paley-Zygmund inequality we get
\[
\Pr\left(\|X\|_{\calz_p(X)}\geq \frac{1}{C}\sqrt{\frac{n}{p}}\right)
\geq \Pr\left(\|X\|_{\calz_p(X)}\geq \frac{1}{2}\Ex \|X\|_{\calz_p(X)} \right)
\geq \frac{(\Ex \|X\|_{\calz_p(X)})^2}{4\Ex \|X\|_{\calz_p(X)}^2}\geq c.
\]
\end{proof}

\begin{lem}
\label{lem:esthalf}
Let $X$ by a symmetric isotropic $n$-dimensional log-concave vector. Then 
$\Ex X_{\lceil n/2\rceil }^*\geq \frac{1}{C}$.
\end{lem}

\begin{proof}
Let $a_i>0$ be such that $\Pr(X_i\geq a_i)=3/8$. Then by the log-concavity of $X_i$, 
$\Pr(|X_i|\geq ta_i)=2\Pr(X_i\geq ta_i)\leq (3/4)^t$ for $t\geq 1$ and integration by parts 
yields $\|X_i\|_2\leq Ca_i$. Thus $a_i\geq c_1$ for a constant $c_1>0$.   

Let $S=\sum_{i=1}^n I_{\{|X_i|\geq c_1\}}$. Then $\Ex S=\sum_{i=1}^n\Pr(|X_i|\geq c_1)\geq 3n/4$. On the other
hand $\Ex S\leq \frac{n}{2}+n\Pr(X_{\lceil n/2\rceil }^*\geq c_1)$, so 
\[
\Ex X_{\lceil n/2\rceil}^*\geq c_1\Pr(X_{\lceil n/2\rceil}^*\geq c_1)\geq c_1/4.
\]
\end{proof}

The next example shows that the tail and moment bounds in Corollary \ref{cor:unclogcon} are optimal.

\medskip

\noindent
{\bf Example.} Let $X=(X_1,\ldots,X_n)$ be an isotropic random vector with i.i.d. symmetric exponential coordinates (i.e. $X$ has the density $2^{n/2}\exp(-\sqrt{2}\|x\|_1)$). Then $(\Ex|X_i|^p)^{1/p}\leq p/2$, so $\frac{2}{p}e_i\in \calm_p(X)$ and
\[
\Pr\left(\|X\|_{\calz_p(X)}\geq t\sqrt{n/p}\right)\geq \Pr(|X_i|\geq t\sqrt{np}/2)\geq e^{-t\sqrt{np}/\sqrt{2}} 
\]  
and for $q=s\sqrt{np}$, $s\geq 1$,
\[
\left(\Ex  \|X\|_{\calz_p(X)}^{q}\right)^{1/q}\geq \frac{2}{p}\|X_i\|_{q}\geq cq/p=cs\sqrt{n/p}.
\]

\section{General case -- approach via entropy numbers}

In this section we propose a method of deriving estimates for $\calz_p$-norms via entropy estimates for 
$\calm_p$-balls and Euclidean distance. We use a standard notation -- for sets $T,S\subset \er^n$, by $N(T,S)$ we denote the minimal number of translates of $S$ that are enough 
to cover $T$. If $S$ is the $\ve$-ball with respect to some translation-invariant metric $d$ then $N(T,S)$ is also denoted as $N(T,d,\ve)$ and is called the metric entropy of $T$ with respect to $d$.

We are mainly interested in log-concave vectors or random vectors which satisfy moment estimates
\begin{equation}
\label{eq:reggrow}
\|\langle t,X\rangle\|_p\leq \lambda \frac{p}{q}\|\langle t,X\rangle\|_q
\quad \mbox{ for }p\geq q\geq 2.
\end{equation}

%%For symmetric log-concave vectors \eqref{eq:reggrow} holds with $\alpha=1$ ($\alpha=2$ for mean zero 
%%log-concave vectors and $\alpha=3$ for arbitrary log-concave vectors).

Let us start with a simple bound.

\begin{prop}
\label{prop:ent_to_Zp}
Suppose that $X$ is isotropic in $\er^n$ and \eqref{eq:reggrow} holds. Then for any $p\geq 2$ and $\ve>0$ we have
\[
\left(\Ex\|X\|_{\calz_p(X)}^2\right)^{1/2}
\leq \ve\sqrt{n}+\frac{e\lambda}{p}\max\left\{p,\log N(\calm_p(X),\ve B_2^n)\right\}.
\]
\end{prop}

\begin{proof}
Let $N=N(\calm_p(X),\ve B_2^n)$ and choose $t_1,\ldots,t_N\in \calm_p(X)$ such that
$\calm_p(X)\subset \bigcup_{i=1}^N (t_i+\ve B_2^n)$. Then
\[
\|x\|_{\calz_p(X)}\leq \ve|x|+\sup_{i\leq N}\langle t_i,x\rangle.
\]
Let $r=\max\{p,\log N\}$. We have 
\begin{align*}
\left(\Ex \sup_{i\leq N}|\langle t_i,X\rangle|^2\right)^{1/2}
&\leq \left(\Ex \sup_{i\leq N}|\langle t_i,X\rangle|^r\right)^{1/r}
\leq \left(\sum_{i=1}^N\Ex|\langle t_i,X\rangle|^r\right)^{1/r}
\\
&\leq N^{1/r}\sup_i\|\langle t_i,X\rangle\|_r
\leq e\lambda\frac{r}{p}\sup_i\|\langle t_i,X\rangle\|_p\leq e\lambda\frac{r}{p}.
\end{align*}
\end{proof}

\begin{rem}
The Paouris inequality \cite{Pa} states that for isotropic log-concave vectors and $q\geq 2$,
$(\Ex|X|^q)^{1/q}\leq C(\sqrt{n}+q)$, so for such vectors and $q\geq 2$,
\[
\left(\Ex\|X\|_{\calz_p(X)}^q\right)^{1/q}
\leq C\ve(\sqrt{n}+q)+\frac{2e}{p}\max\{p,q,\log N(\calm_p(X),\ve B_2^n)\}.
\]
\end{rem}

Unfortunately, the known estimates for entropy numbers of $\calm_p$-balls are rather weak.

\begin{thm}[{\cite[Proposition 9.2.8]{BGVV}}]
\label{thm:entZplogc}
Assume that $X$ is isotropic log-concave and $2\leq p\leq \sqrt{n}$. Then
\[
\log N\left(\calm_p(X),\frac{t}{\sqrt{p}} B_2^n\right)
\leq C\frac{n\log^2 p\log t}{t}
\quad \mbox{ for }1\leq t\leq \min\left\{\sqrt{p},\frac{1}{C}\frac{n\log p}{p^2}\right\}.
\]
\end{thm}

\begin{cor}
Let $X$ be isotropic log-concave, then
\[
\left(\Ex\|X\|_{\calz_p(X)}^p\right)^{1/p}
\leq C\left(\frac{n}{p}\right)^{3/4}\log p\sqrt{\log n}
\quad \mbox{for }2\leq p\leq \frac{1}{C}n^{3/7}\log^{-2/7}n.
\]
\end{cor}

\begin{proof}
We apply  Theorem \ref{thm:entZplogc} with $t=(n/p)^{1/4}\log p\log^{1/2} n$ and Proposition \ref{prop:ent_to_Zp} 
with $\ve=tp^{-1/2}$.
\end{proof}

\begin{rem}
Suppose that $X$ is centered and the following stronger bound than \eqref{eq:reggrow} (satisfied for example for Gaussian vectors) holds
\begin{equation}
\label{eq:gausgrow}
\|\langle t,X\rangle\|_p\leq \lambda \sqrt{\frac{p}{q}}\|\langle t,X\rangle\|_q
\quad \mbox{ for }p\geq q\geq 2.
\end{equation}
Then for any $2\leq p\leq n$,
\[
\frac{1}{\lambda}\sqrt{\frac{2n}{p}}
\leq \left(\Ex\|X\|_{\calz_p(X)}^2\right)^{1/2}\leq
\left(\Ex\|X\|_{\calz_p(X)}^n\right)^{1/n}
\leq 10\lambda \sqrt{\frac{n}{p}}.
\]
\end{rem}

\begin{proof}
Without loss of generality we may assume that $X$ is isotropic. We have
\[
\|\langle t,X\rangle\|_p\leq \lambda\sqrt{p/2}\|\langle t,X\rangle\|_2=\lambda\sqrt{p/2}|t|,
\]
so $\calm_p(X)\supset \lambda^{-1}\sqrt{2/p}B_2^n$ and
\[
\left(\Ex\|X\|_{\calz_p(X)}^2\right)^{1/2}\geq \frac{1}{\lambda}\sqrt{\frac{2}{p}}
\left(\Ex|X|^2\right)^{1/2}=\frac{1}{\lambda}\sqrt{\frac{2n}{p}}.
\]

On the other hand let $S$ be a $1/2$-net in $\calm_p(X)$ of cardinality at most $5^n$. Then
\begin{align*}
\left(\Ex\|X\|_{\calz_p(X)}^n\right)^{1/n}
&\leq 2\left(\Ex\sup_{t\in S}|\langle t,X\rangle|^n\right)^{1/n}
\leq 2\left(\sum_{t\in S}\Ex|\langle t,X\rangle|^n\right)^{1/n}
\\
&\leq 2|S|^{1/n}\sup_{t\in S}\|\langle t,X\rangle\|_n
\leq 10\lambda\sqrt{\frac{n}{p}}\sup_{t\in S}\|\langle t,X\rangle\|_p\leq 10\lambda\sqrt{\frac{n}{p}}.
\end{align*}

\end{proof}

Recall that the Sudakov minoration principle \cite{Su} states that if $G$ is an isotropic Gaussian vector in $\er^n$
then for any bounded $T\subset \er^n$ and $\ve>0$,
\[
\Ex\sup_{t\in T}\langle t,G\rangle\geq\frac{1}{C}\ve\sqrt{\log N(T,\ve B_2^n)}.
\]
So we can say that a random vector $X$ in $\er^n$ satisfies 
\emph{the $L_2$-Sudakov minoration with a constant $C_X$} if for any bounded $T\subset \er^n$ and $\ve>0$,
\[
\Ex\sup_{t\in T}\langle t,X\rangle\geq\frac{1}{C_X}\ve\sqrt{\log N(T,\ve B_2^n)}.
\]

\medskip

\noindent
{\bf Example.}
Any unconditional $n$-dimensional random vector satisfies the $L_2$-Sudakov minoration with constant
$C\sqrt{\log (n+1)}/(\min_{i\leq n} \Ex|X_i|)$.

Indeed, we have by the unconditionality, Jensen's inequality and the contraction principle,
\[
\Ex\sup_{t\in T}\sum_{i=1}^nt_iX_i=\Ex\sup_{t\in T}\sum_{i=1}^nt_i\ve_i |X_i|\geq
\Ex\sup_{t\in T}\sum_{i=1}^nt_i\ve_i \Ex|X_i|\geq \min_{i\leq n}\Ex |X_i|\Ex\sup_{t\in T}\sum_{i=1}^nt_i\ve_i.
\]
On the other hand, the classical Sudakov minoration and the contraction principle yields
\begin{align*}
\frac{1}{C}\ve\sqrt{\log N(T,\ve B_2^n)}
&\leq \Ex\sup_{t\in T}\sum_{i=1}^nt_i g_i\leq
\Ex \max_{i\leq n}|g_i|\Ex\sup_{t\in T}\sum_{i=1}^nt_i \ve_i
\\
&\leq C\sqrt{\log (n+1)}\Ex\sup_{t\in T}\sum_{i=1}^nt_i \ve_i.
\end{align*}

However the $L_2$-Sudakov minoration constant may be quite large in the isotropic case even for
unconditional vectors if we do not assume that $L_1$ and $L_2$ norms of $X_i$ are comparable. Indeed, 
let $\Pr(X=\pm n^{1/2}e_i)=\frac{1}{2n}$ for $i=1,\ldots,n$, where $e_1,\ldots,e_n$ is the canonical
basis of $\er^n$. Then $X$ is isotropic and unconditional.  Let $T=\{t\in \er^n\colon\ \|t\|_\infty\leq n^{-1/2}\}$. Then
\[
\Ex\sup_{t\in T}|\langle t,X\rangle|\leq 1.
\]
However, by the volume-based estimate,
\[
N(T,\ve B_2^n)\geq \frac{\mathrm{vol}(T)}{\mathrm{vol}(\ve B_2^n)}\geq \left(\frac{1}{C\ve}\right)^n,
\]
hence
\[
\sup_{\ve>0}\ve\sqrt{\log N((T,\ve B_2^n)}\geq \frac{1}{C}\sqrt{n}.
\]
Thus the $L_2$-Sudakov constant $C_X\geq \sqrt{n}/C$ in this case.

\medskip

Next proposition shows that random vectors with uniformly log-convex density satisfy the $L_2$-Sudakov minoration.

\begin{prop}
Suppose that a symmetric random vector $X$ in $\er^n$ has the density of the form $e^h$ such that 
$\mathrm{Hess}(h)\geq -\alpha \mathrm{Id}$ for some $\alpha>0$. Then $X$ satisfies the $L_2$-Sudakov minoration with constant $C_X\leq C\sqrt{\alpha}$.
\end{prop}

\begin{proof}
We will follow the method of the proof of the (dual) classical Sudakov inequality (cf. (3.15) and its proof in \cite{LeT}).

Let $T$ be a bounded symmetric set and
\[
A:=\Ex \sup_{t\in T}|\langle t,X\rangle|.
\]

By the duality of entropy numbers \cite{AMS} we need to show that  
$\log^{1/2} N(\ve^{-1}B_2^n,T^o)\leq C\ve^{-1}\alpha^{1/2} A$ for $\ve >0$ or equivalently that 
\begin{equation}
\label{eq:sudestalpha}
N(\delta B_2^n, 6A T^o)\leq \exp(C\alpha \delta^2)
\quad \mbox{ for }\delta >0. 
\end{equation}

To this end let $N=N(\delta B_2^n, 6A T^o)$. If $N=1$ there is nothing to show, so assume that $N\geq 2$. Then we may choose $t_1,\ldots,t_N\in \delta B_2^n$ such
that the balls $t_i+3AT^0$ are disjoint. Let $\mu=\mu_X$ be the distribution of $X$. By the Chebyshev inequality,
\[
\mu(3A T^0)=1-\Pr\left(\sup_{t\in T}|\langle t,X\rangle|>3A\right)\geq \frac{2}{3}.
\] 
Observe also that for any symmetric set $K$ and $t\in \er^n$,
\begin{align*}
\mu(t+K)
&=\int_{K}e^{h(x-t)}dx=\int_{K}e^{h(x+t)}dx=\int_K \frac{1}{2}(e^{h(x-t)}+e^{h(x+t)})dx
\\
&\geq \int_K e^{(h(x-t)+h(x+t))/2}dx.
\end{align*}
By Taylor's expansion we have for some $\theta\in [0,1]$,
\[
\frac{h(x-t)+h(x+t)}{2}
=h(x)
+\frac{1}{4}(\langle\mathrm{Hess}h(x+\theta t)t,t\rangle+\langle\mathrm{Hess}h(x-\theta t)t,t\rangle)\geq
h(x)-\frac{1}{2}\alpha|t|^2.
\]
Thus
\[
\mu(t+K)\geq \int_K e^{h(x)-\alpha|t|^2/2}= e^{-\alpha|t|^2/2}\mu(K)
\]
and
\[
1\geq \sum_{i=1}^N\mu(t_i+3AT^0)\geq \sum_{i=1}^Ne^{-\alpha|t_i|^2/2}\mu(3AT^0)\geq
\frac{2N}{3}e^{-\alpha \delta^2/2}\geq N^{1/3}e^{-\alpha \delta^2/2}
\]
and \eqref{eq:sudestalpha} easily follows.

\end{proof}

\begin{prop}
Suppose that $X$ satisfies the $L_2$-Sudakov minoration with constant $C_X$. Then for any $p\geq 2$
\[
N\left(\calm_p(X),\frac{eC_X}{\sqrt{p}}B_2^n\right)\leq e^p.
\]
In particular if $X$ is isotropic we have for $2\leq p\leq n$,
\[
\left(\Ex \|X\|_{\calz_p(X)}^2\right)^{1/2}\leq e\left(C_X\sqrt{\frac{n}{p}}+1\right).
\]
\end{prop}

\begin{proof}
Suppose that $N=N(\calm_p(X),eC_Xp^{-1/2}B_2^n)\geq e^p$. We can choose
$t_1,\ldots,t_N\in \calm_p(X)$, such that $\|t_i-t_j\|_2\geq eC_Xp^{-1/2}$.
We have
\[
\Ex\sup_{i\geq N}\langle t_i,X\rangle\geq \frac{1}{C_X}eC_Xp^{-1/2}\sqrt{\log N}> e.
\]
However on the other hand,
\[
\Ex\sup_{i\geq N}\langle t_i,X\rangle
\leq \left(\Ex\sup_{i\geq N}|\langle t_i,X\rangle|^p\right)^{1/p}
\leq \left(\sum_{i\geq N}\Ex|\langle t_i,X\rangle|^p\right)^{1/p}
\leq N^{1/p}\max_{i}\|\langle t_i,X\rangle\|_p\leq e.
\]

To show the second estimate we proceed in a similar way as in the proof of Proposition \ref{prop:ent_to_Zp}. 
We choose $T\subset \calm_p(X)$ such that $|T|\leq e^p$
and $\calm_p(X)\subset T+eC_Xp^{-1/2}B_2^n$.
We have
\[
\|X\|_{\calz_p(X)}\leq eC_Xp^{-1/2}|X|+\sup_{t\in T}|\langle t,X\rangle|
\]
so that
\[
\left(\Ex \|X\|_{\calz_p(X)}^2\right)^{1/2}\leq eC_Xp^{-1/2}(\Ex|X|^2)^{1/2}
+\left(\Ex\sup_{t\in T}|\langle t,X\rangle|^2\right)^{1/2}.
\]
Vector $X$ is isotropic, so $\Ex|X|^2=n$ and since $T\subset \calm_p(X)$ and $p\geq 2$ we get
\begin{align*}
\left(\Ex\sup_{t\in T}|\langle t,X\rangle|^2\right)^{1/2}
&\leq
\left(\Ex\sup_{t\in T}|\langle t,X\rangle|^p\right)^{1/p}
\leq 
\left(\sum_{t\in T}\Ex|\langle t,X\rangle|^p\right)^{1/p}
\\
&\leq |T|^{1/p}\max_{t\in T}\||\langle t,X\rangle\|_p\leq e.
\end{align*}

\end{proof}

\noindent
Rafa{\l} Lata{\l}a\\
Institute of Mathematics\\
University of Warsaw\\
Banacha 2\\
02-097 Warszawa\\
Poland\\
{\tt rlatala@mimuw.edu.pl}

\end{document}